\documentclass[a4paper,twoside,9pt]{article}
\usepackage{amsmath,amssymb,graphics,epsfig,color}
\usepackage{amsmath,amsthm,amsfonts,latexsym,amscd,amssymb,enumerate,wasysym,stmaryrd}

\usepackage{graphicx}
\usepackage[all]{xy}
\usepackage{fancyhdr}
\usepackage[T1]{fontenc}
\usepackage{ifthen}
\usepackage{amsmath,amsfonts,amssymb}
\newcommand{\pleinepage}%
{\setlength{\oddsidemargin}{0in}\setlength{\textwidth}{6.26in}\setlength{\topmargin}{0in}\setlength{\textheight}{8.7in}}
\newcommand{\grossepage}%
{\setlength{\oddsidemargin}{-0.5cm}\setlength{\textwidth}{17.5cm}\setlength{\topmargin}{-1.5cm}\setlength{\textheight}{24cm}}
\newcommand{\defit}[1]%
{{\em #1}}

\def\Re{\mathop{\plainRe\mkern -2mu\mit e}\nolimits}
\def\Im{\mathop{\plainIm\mkern -2mu\mit m}\nolimits}
\def\surl#1_#2{\mathrel{\mathop{\kern 0pt #1}\limits_{#2}}}

\newcommand{\fleche}[1]%
{\rTo^{#1}}
\newcommand{\fonction}[5]%
{\begin{diagram}
#2 & {} &\rTo^{#1} & {} & #3 \\
#4 & {} &\rMapsto & {} & #5 
\end{diagram} }
\newcommand{\sfonction}[5]%
{$\begin{array}{ccc}#2 & {\buildrel #1 \over \rightarrow} & #3 \\#4 & \mapsto & #5 \\ \end{array}$ }
\newcommand{\accolade}[1]%
{\begin{cases}  #1 \end{cases}}

\newcounter{nbre}
\newcommand{\entete}[6]%
{{\large \noindent%
\mbox{\begin{tabular}{c} #1 \\ #2 \end{tabular}}\hspace{\fill}\mbox{\begin{tabular}{c} #3 \\ #4 \end{tabular}}\vspace{1cm}\begin{center}{\Huge \textsc{#5}} \\ \vspace{0.5cm}\begin{tabular}{c} #6 \\ \hline \end{tabular}\end{center}\bigskip}}
\newcommand{\defifont}{ \sc }

{\refstepcounter{compteur} \par \vspace{0.3cm}\ \\ \noindent {{  \textsc{\textbf{Definitions}}} 
    \textbf{\thecompteur} \ 
    }---\ \sffamily\renewcommand{\em}{\normalfont\itshape}}{\par
    \vspace{0.3 cm}}

\newcommand{\gloss}[1]%
 {\index{{#1}@{#1}}{\em #1}\relax}
\newcommand{\xgloss}[2]%
 {\index{{#1}@{#1}!{#2}@{#2}}{\em #1\relax #2}\relax}
 \newcommand{\glossref}[2]%
 {\index{{#2}@{#1}}{\em #1}\relax}
 \newcommand{\xglossref}[4]%
 {\index{{#3}@{#1}!{#4}@{#2}}{\em #1 \relax #2 }\relax}
\newcounter{compteur}
\renewcommand{\thecompteur}{\thesection.\arabic{compteur}}
\newenvironment{dfn}[1][]%
{\refstepcounter{compteur} \par \vspace{0.3cm}\ \\ \noindent {{  \textsc{\textbf{Definition}}} 
    \textbf{\thecompteur} \ 
    }---\ \sffamily\renewcommand{\em}{\normalfont\itshape}}{\par
    \vspace{0.3 cm}}
{\refstepcounter{compteur} \par \vspace{0.3cm}\ \\ \noindent {{  \textsc{\textbf{Definitions}}} 
    \thecompteur \ 
    }---\ \sffamily\renewcommand{\em}{\normalfont\itshape}\begin{enumerate}}{\end{enumerate}
    \par \vspace{0.5 cm}}
\newcounter{theonum}

\newenvironment{thm}[1][]%
{\refstepcounter{theonum} \par \vspace{0.3cm}\ \\ \noindent {{  \textsc{\textbf{Theorem}}} 
    \textbf{\thecompteur} \ 
    }---\ \sffamily\renewcommand{\em}{\normalfont\itshape}}{\par
    \vspace{0.3 cm}}

\newcommand\addpage[2]{#2, page #1}
\makeatletter
\renewcommand\p@theonum{\protect\addpage{\thepage}}
\makeatother

{\refstepcounter{compteur} \par \vspace{0.2cm}\ \\ \noindent {{  \textsc{\textbf{Proposition}}} 
    \textbf{\thecompteur} \ 
    }---\ \sffamily\renewcommand{\em}{\normalfont\itshape}}{\par
    \vspace{0.2 cm}}

\newenvironment{prop*}[1][]%
{ \par \vspace{0.2cm}\ \\ \noindent {{\textsc{\textbf{Proposition}}} 
    #1 \ 
    }---\sffamily\renewcommand{\em}{\normalfont\itshape}}{\par \vspace{0.3 cm}}
{\refstepcounter{compteur} \par \vspace{0.3cm}\ \\ \noindent {{  \textsc{\textbf{Properties}}} 
    \thecompteur \ 
    }---\ \sffamily\renewcommand{\em}{\normalfont\itshape}}{\par
    \vspace{0.3 cm}}

{\refstepcounter{compteur} \par \vspace{0.3cm}\ \\ \noindent {{ \textsc{\textbf{Corollary}}}
    \textbf{\thecompteur} 
    \ }---\ \sffamily\renewcommand{\em}{\normalfont\itshape}}{\par  \vspace{0.3 cm}}

{\refstepcounter{compteur} \par \vspace{0.2cm}\ \\ \noindent {{  \textsc{\textbf{Lemma}}} 
    \textbf{\thecompteur} \ 
    }---\ \sffamily\renewcommand{\em}{\normalfont\itshape}}{\par
    \vspace{0.2 cm}}    
    
\renewenvironment{proof}%
{\par \vspace{0.2cm}\ \\ \noindent{ { \textsc{Proof}}\,---\ } }{\hfill{$\Box$} \par \vspace{0.2 cm}}

{\par \vspace{0.2cm}\ \\ \noindent{\sc \textbf{Remark}\ }---\ }{\par \vspace{0.2cm}}
{\refstepcounter{compteur}\par \vspace{0.2cm}\ \\ \noindent{\sc \textbf{Conjecture} \textbf{\thecompteur}\ }---\ }{\par \vspace{0.2cm}}

\begin{document}

\NoCompileMatrices
\def\ds{\displaystyle}
\def\pn{\pi_{n}}
\def\pnu{\pi_{n-1}}
\pagestyle{fancy}
\renewcommand{\sectionmark}[1]{\markright{\thesection\ #1}}
\fancyhf{} 


\renewcommand{\headrulewidth}{0.16pt}
\renewcommand{\footrulewidth}{0pt}
\addtolength{\headheight}{0.7pt} 
\renewcommand{\headrulewidth}{0pt} 

\newcommand{\foot}[1]{\footnote{\begin{normalsize}#1\end{normalsize}}}


\def\bX{\partial X}
\def\dim{\mathop{\rm dim}}
\def\Re{\mathop{\rm Re}}
\def\Im{\mathop{\rm Im}}
\def\I{\mathop{\rm I}}
\def\Id{\mathop{\rm Id}}
\def\grad{\mathop{\rm grad}}
\def\vol{\mathop{\rm vol}}
\def\SU{\mathop{\rm SU}}
\def\SO{\mathop{\rm SO}}
\def\Aut{\mathop{\rm Aut}}
\def\End{\mathop{\rm End}}
\def\GL{\mathop{\rm GL}}
\def\Cinf{\mathop{\mathcal C^{\infty}}}
\def\Ker{\mathop{\rm Ker}}
\def\Coker{\mathop{\rm Coker}}
\def\dom{\mathop{\rm Dom}}
\def\Hom{\mathop{\rm Hom}}
\def\Ch{\mathop{\rm Ch}}
\def\sign{\mathop{\rm sign}}
\def\SF{\mathop{\rm SF}}
\def\loc{\mathop{\rm loc}}
\def\AS{\mathop{\rm AS}}
\def\spec{\mathop{\rm spec}}
\def\Ric{\mathop{\rm Ric}}
\def\ch{\mathop{\rm ch}}
\def\Ch{\mathop{\rm Ch}}

\def\ev{\mathop{\rm ev}}
\def\id\textrm{Id}
\def\dd{\mathcal{D}(d)}
\def\Cli{\mathbb{C}l(1)}
\def\kerd{\operatorname{ker}(d)}

\def\Fi{\Phi}

\def\de{\delta}
\def \dl{\partial L_x^0}
\def\e{\eta}
\def\ep{\epsilon}
\def\ro{\rho}
\def\a{\alpha}
\def\o{\omega}
\def\O{\Omega}
\def\b{\beta}
\def\la{\lambda}
\def\th{\theta}
\def\s{\sigma}
\def\t{\tau}
\def\g{\gamma}
\def\D{\Delta}
\def\G{\Gamma}
\def \fol{\mathcal F}
\def\R{\mathbin{\mathbb R}}
\def\Rn{\R^{n}}
\def\C{\mathbb{C}}
\def\Cm{\mathbb{C}^{m}}
\def\Cn{\mathbb{C}^{n}}
\def\gr{\mathcal{G}}
\def\Kahler{{K\"ahler}}
\def\w{{\mathchoice{\,{\scriptstyle\wedge}\,}{{\scriptstyle\wedge}}
{{\scriptscriptstyle\wedge}}{{\scriptscriptstyle\wedge}}}}
\def\cA{{\cal A}}\def\cL{{\cal L}}
\def\cO{{\cal O}}\def\cT{{\cal T}}\def\cU{{\cal U}}
\def\cD{{\cal D}}\def\cF{{\cal F}}\def\cP{{\cal P}}\def\cH{{\cal H}}\def\cL{{\cal L}}
\def\cB{{\cal B}}


\newcommand{\n}[1]{\left\| #1\right\|}

\def\Z{\mathbb{Z}}
\def\cgs{C^{*}(\Gamma,\sigma)}
\def\bcgs{C^{*}(\Gamma,\bar{\sigma})}
\def\cgsr{C^{*}_{red}(,\sigma)}
\def\Mt{\tilde{M}}
\def\Et{\tilde{E}}
\def\Vt{\tilde{V}}
\def\Xt{\tilde{X}}
\def\N{\mathbb{N}}
\def\Nbs{\N^{\bar{\s}}}
\def\rcab{\ro^{[c]}_{\a-\b}}
\def\rc{\ro^{[c]}}
\def\Cd{\mathbb{C}^{d}}
\def\tr{\mathop{\rm tr}}\def\tralg{\tr{}^{\text{alg}}}       

\def\TR{\mathop{\rm TR}}\def\trace{\mathop{\rm trace}}
\def\STR{\mathop{\rm STR}}
\def\trG{\mathop{\rm tr_\Gamma}}
\def\TRG{\mathop{\rm TR_\Gamma}}
\def\Tr{\mathop{\rm Tr}}
\def\Str{\mathop{\rm Str}}
\def\Cl{\mathop{\rm Cl}}
\def\Op{\mathop{\rm Op}}
\def\supp{\mathop{\rm supp}}
\def\scal{\mathop{\rm scal}}
\def\ind{\mathop{\rm ind}}
\def\Ind{\mathop{\mathcal I\rm nd}\,}
\def\Diff{\mathop{\rm Diff}}
\def\T{\mathcal{T}}
\def\dn{\textrm{dim}_{\Lambda}}
\def \lke{\textrm L^2-\textrm{Ker}}


\newcommand{\wt}{\widetilde}
\newcommand{\go}{\mathcal{G}^{0}}
\newcommand{\dii}{(d_x^{k-1})^\ast}
\newcommand{\di}{d_x^{k-1}}
\newcommand{\ra}{\operatorname{range}}
\newcommand{\rb}{\rangle}
\newcommand{\lb}{\langle}
\newcommand{\re}{\mathcal{R}}
\newcommand{\vo}{\operatorname{End}_{\Lambda}(E)}
\newcommand{\mt}{\mu_{\Lambda,T}}
\newcommand{\tru}{\operatorname{tr}_{\Lambda}}
\newcommand{\buno}{B^1_{\Lambda}(E)}
\newcommand{\bdue}{B^2_{\Lambda}(E)}
\newcommand{\clis}{H^{2k}_{(2),dR}(L_x^0)}
\newcommand{\cali}{L^2(\Omega^{2k}(\partial L_x^0))}
\newcommand{\binf}{B^{\infty}_{\Lambda}(E)}
\newcommand{\bif}{B^{f}_{\Lambda}(E)}
\newcommand{\vn}{\operatorname{End}_{\mathcal{R}}}
\newcommand{\ho}{\operatorname{Hom}_{\Lambda}}
\newcommand{\spc}{\operatorname{spec}_{\Lambda,e}}
\newcommand{\ix}{\operatorname{Ind}_{\Lambda}}
\newcommand{\cic}{C^{\infty}_c(L_x;E_{|L_x})}
\newcommand{\ci}{C^{\infty}_c(L_x;E_{|L_x})}
\newcommand{\tx}{\{T_x\}_{x\in X}}
\newcommand{\cc}{C^{\infty}_c(X)}
\newcommand{\rom}{\underline{\mathcal{R}_0}}
\newcommand{\roma}{(\mathcal{R}_0)_{|\partial X_0}     }
\newcommand{\dfo}{D^{\mathcal{F}_{\partial}}}
\newcommand{\deu}{D_{\epsilon,u}}
\newcommand{\deupp}{D^{+}_{\epsilon,u}}
\newcommand{\deum}{D^{-}_{\epsilon,u}}
\newcommand{\deuf}{D_{\epsilon,u}^{\mathcal{F}_{\partial}}}
\newcommand{\deufo}{D_{\epsilon,u,x_0}^{\mathcal{F}_{\partial}}}
\newcommand{\pie}{\Pi_{\epsilon}}
\newcommand{\pal}{\partial L_x}
\newcommand{\pr}{\partial_r}
\newcommand{\inbl}{\int_{\partial L_x} }
\newcommand{\pkp}{\chi_{\{0\}}(D^+_x)}
\newcommand{\deup}{D_{\epsilon,x}^{\pm}}
\newcommand{\dext}{D_{\epsilon,\mp u,x}^{\pm}}
\newcommand{\dex}{D_{\epsilon,\pm u,x}^{\pm}}
\newcommand{\ext}{\operatorname{Ext}(D_{\epsilon,x}^{\pm})}
\newcommand{\eppu}{0<|u|<\epsilon}
\newcommand{\hdeupx}{e^{-tD^2_{\epsilon,u,x}}}
\newcommand{\udif}{\operatorname{UDiff}}
\newcommand{\ki}{L^2(\Omega^kL_x^0)}
\newcommand{\uc}{\operatorname{UC}}
\newcommand{\op}{\operatorname{Op}}
\newcommand{\deux}{D_{\epsilon,u,x}}
\newcommand{\pk}{\phi_k}
\newcommand{\hdeupsx}{e^{-sD^2_{\epsilon,u,x}}}
\newcommand{\hdeups}{e^{-sD^2_{\epsilon,u}}}
\newcommand{\hdeut}{e^{-tD^2_{\epsilon,u}}}
\newcommand{\indu}{\operatorname{ind}_{\Lambda}}
\newcommand{\stru}{\operatorname{str}_{\Lambda}}
\newcommand{\deuq}{D^2_{\epsilon,u}}
\newcommand{\intk}{\int_{\sqrt{k}}^{\infty}}
\newcommand{\dmd}{d\mu_{\Lambda,D_{\epsilon,u}}(x)}
\newcommand{\defox}{D_{x}^{\mathcal{F}_{\partial}}}
\newcommand{\mun}{\mu_{\Lambda,D_{\epsilon,u}}(x)}
\newcommand{\tsi}{\int_{-\sigma}^{\sigma}}
\newcommand{\ak}{\lim_{k\rightarrow \infty}\operatorname{LIM}_{s\rightarrow 0}}
\newcommand{\deus}{D_{\epsilon,u}e^{-tD_{\epsilon,u}^2}}

\newcommand{\deuss}{D_{\epsilon,u}^2e^{-tD_{\epsilon,u}^2}}
\newcommand{\pkd}{\phi_k^2}
\newcommand{\eup}{e^{-tD^{+}_{\epsilon,u}D^{-}_{\epsilon,u}}}
\newcommand{\eum}{e^{-tD^{-}_{\epsilon,u} D^{+}_{\epsilon,u} }}
\newcommand{\deussx}{D_{\epsilon,u,x}e^{-tD_{\epsilon,u,x}^2}}
\newcommand{\dessx}{S_{\epsilon,u,x}e^{-tS_{\epsilon,u,x}^2}}
\newcommand{\clib}{c(\partial_r)\partial_r \phi_k^2}
\newcommand{\sk}{\int_s^{\sqrt{k}}}
\newcommand{\esm}{S_{\epsilon,u}e^{-tS_{\epsilon,u}^2}}
\newcommand{\deussxo}{D_{\epsilon,u,z_0}e^{-tD_{\epsilon,u,x_0}^2}}
\newcommand{\dessxo}{S_{\epsilon,u,z_0}e^{-tS_{\epsilon,u,z_0}^2}}
\newcommand{\deusszo}{D^{\mathcal{F}_{\partial}}_{\epsilon,u,z_0}e^{-t(D^{\mathcal{F}_{\partial}}_{\epsilon,u,x_0})2}}
\newcommand{\desszo}{S_{\epsilon,u,x_0}e^{-tS_{\epsilon,u,x_0}^2}}
\newcommand{\essp}{S_{\epsilon,u}^2}
\newcommand{\esspo}{S_{0,u}^2}
\newcommand{\dotto}{\dot{\theta}}
\newcommand{\piep}{\Pi_{\epsilon}}
\newcommand{\ome}{\Omega}
\newcommand{\deffo}{D^{\mathcal{F}_{\partial}}}
\newcommand{\nablal}{\nabla_x^l}
\newcommand{\nablak}{\nabla_y^k}
\newcommand{\kerk}{[f(P)]_{(x_0,\bullet)} }
\newcommand{\kepp}{\operatorname{Ker} (D^{\mathcal{F}_0^+})}
\newcommand{\ty}{\infty}
\definecolor{light}{gray}{.95}
\newcommand{\pecetta}[1]{
$\phantom .$
\bigskip
\par\noindent
\colorbox{light}{\begin{minipage}{13.5 cm}#1\end{minipage}}
\bigskip
\par\noindent
}

\newcommand\Di{D\kern-7pt/}

\title{Complex Lie Algebroids and ACH manifolds}
\author{\Large Paolo Antonini\\
Paolo.Antonini@mathematik.uni-regensburg.de
\\
paolo.anton@gmail.com}

\maketitle
\begin{abstract}We propose the definition of a Manifold with a complex Lie structure at infinity. The important class of ACH manifolds enters into this class.
\end{abstract}

\tableofcontents
















\section{Manifolds with a complex Lie structure at infinity}
Manifolds with a Lie structure at infinity are well known in literature \cite{Ammann,Ammann2,Ammann3}.  Remind the definition. Let $X$ be a smooth non compact manifold together with a compactification $X\hookrightarrow \overline{X}$ where $\overline{X}$ is a compact manifold with corners\footnote{with embedded hyperfaces i.e the definition requires that every boundary hypersurface has a smooth defining function }. A Lie structure at infinity on $X$ is the datum of a Lie subalgebra $\mathcal{V}$ of the Lie algebra of vector fields on $\overline{X}$ subjected to two restrictions
\begin{enumerate}
\item every vector field in $\mathcal{V}$ must be tangent to each boundary hyperface of $\overline{X}$.
\item $\mathcal{V}$ must be a finitely generated $C^{\infty}(\overline{X})$--module this meaning that exists a fixed number $k$ such that around each point $x\in \overline{X}$ we have for every $V\in \mathcal{V}$,
$$\varphi(V-\sum_{i_1}^k \varphi_k V_k)=0$$
 where $\varphi$  is a function with $\varphi=1$ in the neighborhood, the vector fields $V_1,...,V_n$ belong to $\mathcal{V}$ and the coefficients $\varphi_j$ are smooth functions with univoquely determined germ at $x$. 
\end{enumerate} By The Serre--Swann equivalence there must be a Lie Algebroid over $\overline{X}$ i.e. a smooth vector bundle $A\longrightarrow \overline{X}$ with a Lie structure on the space of sections $\Gamma(A)$ and a vector bundle map $\rho:A\longrightarrow \overline{X}$ such that the extended map on sections is a morphism of Lie algebras and satisfies 
\begin{enumerate}
\item $\rho(\Gamma(A))={A}$
\item $[X,fY]=f[X,Y]+(\rho(X)f)Y$ for all $X,Y\in \Gamma(A)$.
\end{enumerate}
In particular we can define manifolds with a Lie structure at infinity as manifolds $X$ with a Lie algebroid over a compactification $\overline{X}$ with the image of $\rho$ contained in the space of boundary vector fields (these are called boundary Lie algebroids). Notice that the vector bundle $A$ can be "physically reconstructed" in fact its fiber $A_x$ is naturally the quotient $\mathcal{V}/\mathcal{V}_x$ where 
$$\mathcal{V}_x:=\Big{\{}V\in \mathcal{V}:V=\sum_{\textrm{finite}} \varphi_jV_j,\,V_j\in \mathcal{V}, \,\varphi_j\in C^{\infty}(\overline{X}),\,\varphi_j(x)=0\Big{\}}.$$ In particular since over the interior $X$ there are no restrictions $\rho:A_{|X}\longrightarrow TX$ is an isomorphism. In mostly of the applications this map degenerates over the boundary. One example for all is the Melrose $b$--geometry \cite{melrose} where one takes as $\overline{X}$ a manifold with boundary and $\mathcal{V}$ is the space of all vector fields that are tangent to the boundary. Here $A={}^bT\overline{X}$ the $b$--tangent bundle. In fact all of these ideas are a formalization of long program of Melrose.

In this section we aim to take into account complex Lie algebroids i.e. complex vector bundles with a structure of a complex Lie algebra on the space of sections and the anchor mapping ($\mathbb{C}$--linear, of coarse) with values on the complexified tangent space 
$T_{\mathbb{C}}\overline{X}=T\overline{X}\otimes \mathbb{C}.$
\begin{dfn}
A manifold with a complex Lie structure at infinity is a triple $(X,\overline{X},A)$ where $X\hookrightarrow \overline{X}$ is a compactification with a manifold with corners and $A\longrightarrow \overline{X}$ is a complex Lie algebroid with the $\mathbb{C}$--linear anchor mapping $\rho:A\longrightarrow T_{\mathbb{C}}\overline{X}$ with values on the space of complex vector fields tangent to each boundary hypersurface.
\end{dfn}
Note that over the interior the algebroid $A$ reduces to the complexified tangent bundle so a hermitian metric along the fibers of $A$ restricts to a hermitian metric on $X$. We shall call the corresponding object a {\bf{hermitian manifold with a complex Lie structure at infinity}} or a hermitian Lie manifold.
\section{ACH manifolds}
The achronim ACH stands for asymptotically complex
hyperbolic manifold. This is an important class of non--compact Riemannian manifolds 
 and are strictly related to some solutions of the Einstein equation
 \cite{Biq,Biq2}
  and CR geometry \cite{Biq3}. We are going to remind the definition. Let $\overline{X}$ be a compact manifold of even dimension $m = 2n$ with boundary
$Y$ . We will denote by $X$ the interior of $X$, and choose a defining function $u$
of $Y$ , that is a function on $\overline{X}$, positive on $X$ and vanishing to first order on
$Y = \partial \overline X$.
The notion of ACH metric on $X$ is related to the data of a strictly pseudoconvex
CR structure on $Y$, that is an almost complex structure $J$ on a
contact distribution of $Y$, such that $\gamma(\cdot, \cdot) = d\eta(\cdot, J\cdot)$ is a positive Hermitian
metric on the contact distribution (here we have chosen a contact form $\eta$).
Identify a collar neighborhood of $Y$ in $X$ with $ [0, T)\times Y$ , with coordinate
$u$ on the first factor. A Riemannian metric $g$ is defined to be an ACH metric
on $X$ if there exists a CR structure $J$ on $Y$, such that near $Y$
\begin{equation}
g\sim \dfrac{du^2+\eta^2}{u^2}+\dfrac{\gamma}{u}.
\end{equation} The asymtotic $\sim$ should be intended in the sense that the difference between $g$ and the model metric $g_0=\dfrac{du^2+\eta^2}{u^2}+\dfrac{\gamma}{u}$ is a symmetric $2$--tensor $\kappa$ with $|\kappa|=O({u}^{\delta/2})$, $0<\delta \leq 1$. One also requires that each $g_0$--covariant derivative of $\kappa$ must satisfy $|\nabla^m \kappa|=O({u}^\delta/2)$. The complex structure on the Levi distribution $H$ on the boundary is called {\bf{the conformal infinity}} of $g$. Hereafter we shall take the normalization $$\delta=1.$$ This choice is motivated by applications to the ACH Einstein manifolds where well known normalization results show its naturality \cite{Biq}.

\subsection{The square root of a manifold with boundary}In order to show that ACH manifolds are complex Lie manifolds we need a construction of Melrose, Epstein and Mendoza \cite{Me}. So let $\overline{X}$ be a manifold with boundary with boundary defining function $u$. 
Let us extend the ring of smooth functions $C^{\infty}(\overline{X})$ by  adjoining the function $\sqrt{u}$. Denote this new ring $C^{\infty}(\overline{X}_{1/2})$ In local coordinates a function is in this new structure if it can be expressed as a $C^{\infty}$ function of 
$u^{1/2},y_1,...,y_n$ i.e. it is $C^{\infty}$ in the interior and has an expansion at $\partial{\overline{X}}$ of the form 
$$f(u,x)\sim \sum_{j=0}^{\infty}u^{j/2}a_j(x)$$ with coefficients 
$a_j(x)$ smooth in the usual sense. The difference 
$f-\sum_{j=0}^{N}u^{j/2}a_j(x)$ becomes increasingly smooth with $N$. In this way $f$ is determined by the asymtotic series up to a function with all the derivatives that vanish at the boundary. Since the ring is independent from the choice of the defining function and invariant under diffeomorphisms of $\overline{X}$ the manifold 
$\overline{X}$ equipped with 
$C^{\infty}(\overline{X}_{1/2})$ is a manifold with boundary globally diffeomorphic to $\overline{X}$.
\begin{dfn}The square root of $\overline{X}$ is the manifold $\overline{X}$ equipped with the ring of functions $C^{\infty}(\overline{X}_{1/2})$. We denote it $\overline{X}_{1/2}$
\end{dfn}
Notice the natural mapping $\iota_{1/2}:\overline{X}\longrightarrow \overline{X}_{1/2}$ descending from the inclusion 
 $C^{\infty}(\overline{X})\hookrightarrow C^{\infty}(\overline{X}_{1/2})
$ is not a $C^{\infty}$ isomorphism since it cannot be smoothly inverted. Note also the important fact that the interiors and boundaries of $\overline{X}$ and $\overline{X}_{1/2}$ are canonically diffeomorphic. The change is the way the boundary is attached.
\subsection{The natural complex Lie algebroid associated to an ACH manifold} Let $X$ be an orientable $2n$--dimensional ACH manifold with compactification $\overline{X}$, define $Y:=\partial \overline{X}$ and remember for further use it is canonically diffeomorphic to the boundary of $\overline{X}_{1/2}$. So $Y$ is a CR $(2n-1)$-- manifold with contact form $\eta$ (we keep all the notations above).
Let $H=\operatorname{Ker}\eta$ the Levi distribution with choosen complex structure $J:H\longrightarrow H$. Extend $J$ to a complex linear endomorphism $J:T_{\mathbb{C}}Y\longrightarrow T_{\mathbb{C}}Y$ with $J^2=-1$. Define the complex subundle $T_{1,0}$ of $T_{\mathbb{C}}Y$ as the bundle of the $i$--eigenvectors. Notice that directly from the definition on the CR structure it is closed under the complex bracket of vector fields; for this reason the complex vector space
$$\mathcal{V}_{1,0}:=\{V\in \Gamma(\overline{X}_{1/2},T\overline{X}_{1/2}):V_{|Y}\in \Gamma(T_{1,0})\}$$ is a complex Lie algebra. It is also a finitely generated projective module. To see this, around a point $x\in Y$ let $U_1,...,U_r$, $r=2(n-1)$ span $H$ and let $T\in \Gamma(Y,TY)$ be the Reeb vector field, univoquely determined by the conditions $\gamma(T)=1$ and 
$d\gamma(\cdot,T)=0$. Then it is easy to see that the following is a local basis of $\mathcal{V}_{1,0}$ over 
$C^{\infty}(\overline{X}_{1/2},\mathbb{C})$:
\begin{equation}\label{1}\sqrt{u}\partial_u,\,\,U_1-iJU_1,\,\,...,\,\,U_r-iJU_r,\,\,\sqrt{u}T\end{equation} where $u$ is a boundary defining function. Now let $$\widetilde{\mathcal{V}}_{\textrm{ACH}}:=\sqrt{u}\mathcal{V}_{1,0}$$ the submodule defined by the multiplication of every vector field by the smooth function $\sqrt{u}$. A local basis corresponding to \eqref{1} is 
\begin{equation}
\label{2}
u\partial_u,\,\,\sqrt{u}[U_1-iJU_1],\,\,...,\,\,\sqrt{u}[U_r-iJU_r],\,\,uT.
\end{equation}
 Let $A\longrightarrow \overline{X}_{1/2}$ the corresponding Lie algebroid. The following result is immediate
\begin{thm}
Every ACH metric on $X$ 
 extends to a smooth hermitian metric on $A$. In particular an ACH manifold is a manifold with a Complex Lie structure at infinity.
\end{thm}
\begin{proof}
Just write the matrix of the difference $\kappa$ on a frame of the form
\eqref{2}. This gives the right asymptotic.
\end{proof}

\end{document}